\documentclass[reqno]{amsart}
\usepackage{hyperref}
\usepackage{latexsym,amssymb,amsfonts}
\usepackage{enumerate}
\usepackage{amsthm}
\usepackage{amsmath}

\begin{document}
\title[]
{Locality and Domination of Semigroups}

\author[]
{Khalid Akhlil}

\address{\newline Akhlil Khalid \newline
Applied Analysis Institute
\newline University of Ulm\newline Ulm,
Germany}

\address{\newline
Polydisciplinary Faculty of Ouarzazate\newline Ibn Zohr
University\newline Ouarzazate, Morocco.}

\email{k.akhlil@uiz.ac.ma}

\dedicatory{Dedicated to O. El-Mennaoui and W. Arendt}

\date{\today}
\thanks{The author was supported, for this work, by ``Deutscher
Akademischer Austausch Dienst``(German Academic Exchange Service)
Grant Number: A/11/97482.} 

\begin{abstract}
We characterize all semigroups $(T(t))_{t\geq0}$ on $L^2(\Omega)$
sandwiched between Dirichlet and Neumann ones, i.e.:
\begin{equation*}\label{eq:san}
 e^{t\Delta_D}\leq T(t)\leq e^{t\Delta_N}\quad,\text{for all }t\geq0
\end{equation*}
in the positive operators sense. The proof uses the well-known
Beurling-Deny and Lejan formula to drop the locality assumption made
usually on the form associated with $(T(t))_{t\geq 0}$.
 \end{abstract}

\maketitle \setlength{\textheight}{19.5 cm}
\setlength{\textwidth}{12.5 cm}
\newtheorem{theorem}{Theorem}[section]
\newtheorem{lemma}[theorem]{Lemma}
\newtheorem{proposition}[theorem]{Proposition}
\newtheorem{corollary}[theorem]{Corollary}
\theoremstyle{definition}
\newtheorem{definition}[theorem]{Definition}
\newtheorem{example}[theorem]{Example}
\theoremstyle{remark}
\newtheorem{remark}[theorem]{Remark}
\numberwithin{equation}{section} \setcounter{page}{1}

%%%%%%%%%%%%%%%%%%%%%%%%%%%%%%%%%%%%%%%%%%%%%%%
\newcommand{\ombar}{\overline{\Omega}}
 \newcommand{\po}{\partial\Omega}
 \newcommand{\me}{\mathcal E}
 \newcommand{\mf}{\mathcal F}
 \newcommand{\Huntild}{\widetilde{H}^1(\Omega)}
\newcommand{\emu}{e^{-t\Delta_{\mu}}}
\newcommand{\eN}{e^{-t\Delta_{N}}}
\newcommand{\eD}{e^{-t\Delta_{D}}}
\newcommand{\dmu}{\Delta_{\mu}}
\newcommand{\RR}{\mathbb R}
\newcommand{\Hun}{H^1(\Omega)}
\newcommand{\rcap}{\mathrm{Cap}_{\ombar}}

\section{Introduction}
Let $\Omega$ be an open set of $\RR^d$. We want to characterize all
semigroups $(T(t))_{t\geq0}$ on $L^2(\Omega)$ sandwiched between
Dirichlet and Neumann ones, i.e.:
\begin{equation}\label{eq:san}
 e^{t\Delta_D}\leq T(t)\leq e^{t\Delta_N}\quad,\text{for all }t\geq0
\end{equation}
in the positive operators sense. In \cite{AW2}, it is proved that
when  the (regular) Dirichlet form $(a,D(a))$ associated with such
semigroups is \textit{local} then there exists a unique positive
measure on the boundary $\Gamma:=\partial\Omega$ which charges no
set of zero relative capacity such that the form is given explicitly
by
\begin{equation}
 a(u,v)=\int_{\Omega}\nabla u\nabla v dx+\int_{\Gamma}
 \widetilde{u}\widetilde{v}d\mu,\quad D(a_{\mu})=
 \widetilde{H^1}(\Omega)\cap L^2(\Gamma,d\mu)
\end{equation}

The regularity of $a$ is needed to ensure the uniqueness of the
measure $\mu$, otherwise the form will be regular just on a subset
$X=\overline{\Omega}\setminus\Gamma_{\infty}$, where
$\Gamma_{\infty}$ is the part of $\Gamma$ where $\mu$ is locally
infinite everywhere, and then $\mu$ can not be unique because
obtained by extension to the whole boundary.

In this paper we focus on the hypothesis of locality and to simplify
we consider $a$ as regular Dirichlet form (if not we do the same as
in \cite{AW2} or \cite{Wa}). The motivation of our main results was
an intuitive remark that one don't need to suppose the locality of
$a$, and that the locality is contained in the fact that $a$ is
sandwiched between two local forms. In fact and as mentioned in
\cite[Remark 6.12]{DGK} one can remark that the semigroup considered
in the Example 4.5 in \cite{AW2} can not be positive(actually it is
eventually positive). It follows that one can not make use of
Ouahabaz's Characterization theorem of semigroup's domination.

This work is organized as follow: In section 2 we recall the
framework and the results of \cite{AW2}. In section 3 we discuss the
example 4.5 in \cite{AW2} and we show why it is not positive. In
section 4 we prove our main result using beurling-Deny and Lejan
(BDL) formula.

\section{Preliminaries}

First of all, we recall the situation treated in \cite{AW2}. Let
$\Omega$ be an open set and $\Gamma:=\partial\Omega$. Let
$\mu:\mathcal B\to[0,\infty]$ be a measure where $\mathcal B$
denotes the $\sigma-$algebra of Borel sets of $\Gamma$. We now
consider the symmetric form $a_{\mu}$ on $L^2(\Omega)$ given by
\[
 a_{\mu}(u,v)=\int_{\Omega}\nabla u\nabla v dx+\int_{\Gamma}u(x)v(x) d\mu
\]
with domain
\[
 D(a_{\mu})=\{u\in\Hun\cap C_c(\ombar):\int_{\Gamma}|u|^2d\mu<\infty\}
\]
Let
\[
 \Gamma_{\mu}=\{z\in\Gamma:\exists r>0 \text{ such that }\mu(\Gamma\cap
 B(z,r))<\infty\}
\]be the part of $\Gamma$ on which $\mu$ is locally finite.

Assume now that $\Gamma_{\mu}\neq\emptyset$, thus $\Gamma_{\mu}$ is
a locally compact space and $\mu$ is a regular Borel measure on
$\Gamma_{\mu}$. We say that $\mu$ is \textbf{admissible} if for each
Borel set $A\subset\Gamma_{\mu}$ one has
\[
 \rcap(A)=0\Rightarrow\mu(A)=0
\]where $\rcap(A)$ refer to the relative capacity.

It is proved in \cite{AW2} that the form $a_{\mu}$ is closable if
and only if $\mu$ is admissible, and the closure of $a_{\mu}$ is
given by
\[
 D(\overline{a}_{\mu})=\{u\in\Huntild:\tilde{u}=0\text{ r.q.e. on }
 \Gamma\setminus\Gamma_{\mu},\int_{\Gamma_{\mu}}|\tilde{u}|^2d\mu<\infty\}
\]
\[
 \overline{a}_{\mu}(u,v)=\int_{\Omega}\nabla u\nabla v dx+
 \int_{\Gamma_{\mu}}\tilde{u}\tilde{v}d\mu
\]
Here $\tilde{u}$ is the relatively quasi-continuous representative
of $u$. Note that the form $(a_{\mu},D(a_{\mu})$ is regular if and
only if $\mu$ is a radon measure on $\Gamma$, which means that
$\Gamma_{\mu}=\Gamma$.

Denote by $\Delta_{\mu}$ the operator associated with
$\overline{a}_{\mu}$. Then it follows that $\Delta_{\mu}$ is a
realization of the Laplacian in $L^2(\Omega)$( see \cite{AW1} and
\cite{Wa} for more details and other properties).

Next we give two trivial examples of measure $\mu$:

\begin{example}

(i) If $\mu=0$ then $D(\overline{a}_{\mu})=\Huntild$ and
$\overline{a}_{\mu}(u,v)=\int_{\Omega}\nabla u\nabla v dx$. Let
$\Delta_N$ be the operator associated with $(a_0,D(a_0))$. It is the
Neumann Laplacian, and it coincides with the usual Neumann Laplacian
when $\Omega$ is bounded with Lipschitz boundary.

(ii) If $\Gamma_{\mu}=\emptyset$, then
$D(\overline{a}_{\mu})=H^1_0(\Omega)$ and
$\overline{a}_{\mu}(u,v)=\int_{\Omega}\nabla u\nabla v dx$. Let
$\Delta_D$ be the operator associated with
$(a_{\infty},D(a_{\infty}))$. It is the Dirichlet Laplacian.
\end{example}

We have the following domination result

\begin{theorem}
 For each admissible measure $\mu$, the semigroup
 $(e^{t\Delta_{\mu}})_{t\geq 0}$ satisfies
\begin{equation}\label{eq:san2}
 e^{t\Delta_D}\leq e^{t\Delta_{\mu}} \leq e^{t\Delta_N}
\end{equation}for all $t\geq 0$ in the sense of positive operators.
\end{theorem}

The proof is a simple application of the Ouhabaz criterion and is
given in \cite[Theorem 3.1.]{AW1}. For a probabilistic proof see
\cite{A}.

In \cite[Section 4]{AW1}, W. Arendt and M. Warma explored the
converse of the above Theorem. More precisely, they answered the
following question: Having a sandwiched semigroup between Dirichlet
and Neumann semigroups, can one write the associated form with help
of an admissible measure? The answer was affirmative under a central
hypothesis of locality of the associated Dirichlet form.

The first Theorem says (\cite[Theorem 4.1]{AW1})

\begin{theorem}
 Let $\Omega$ be an open subset of $\RR^d$ with boundary $\Gamma$. Let
 $T$ be a symmetric $C_0-$semigroup on $L^2(\Omega)$ associated with a
 positive closed form $(a,D(a))$. Then the following assertions are equivalent

\begin{enumerate}
 \item[(i)] There exist an admissible measure $\mu$ such that $a=a_{\mu}$
 \item[(ii)] \begin{enumerate}\item[(a)] One has  $e^{t\Delta_D}\leq T(t)\leq
 e^{t\Delta_N}$, $(t\geq0)$.

     \item[(b)] $a$ is local.

     \item[(c)] $D(a)\cap C_c(\ombar)$ is dense in $(D(a),\|.\|_a)$.
     \end{enumerate}
  \end{enumerate}

\end{theorem}

In order to characterize those sandwiched semigroups which come from
bounded measure we have the following Theorem (\cite[Theorem
4.2.]{AW1})

\begin{theorem}

  Let $\Omega$ be a bounded open set of  $\RR^d$. Let $T$ be a symmetric
   $C_0-$semigroup on $L^2(\Omega)$ associated with a positive closed
   form $(a,D(a))$. Then the following assertions are equivalent

  \begin{enumerate}
   \item[(i)] There exist a bounded admissible measure $\mu$ on
   $\Gamma$ such that $a=a_{\mu}$

   \item[(ii)] \begin{enumerate}
   \item[(a)] One has  $e^{t\Delta_D}\leq T(t)\leq e^{t\Delta_N}$
   , $(t\geq0)$.

     \item[(b)] $a$ is local.

     \item[(c)] $1\in D(a)$.
     \end{enumerate}
 \end{enumerate}
\end{theorem}

To end this section we give the following fascinating Beurling-Deny
and Lejan formula. The proof can be found for example in \cite{FOT}
or \cite{BD}, another proof based on contraction operators can be
found in \cite{Al} or \cite{An}.

\begin{theorem}(The Beurling-Deny and Lejan formula)
Any regular Dirichlet form $(a,D(a))$ on $L^2(X;m)$ can be expressed
for $u,v\in D(a)\cap C_c(X)$ as follow:
\begin{multline*}
     a(u,v)=a^{(c)}(u,v)+\int_Xu(x)v(x)k(dx)\\
     +\int_{X\times X\setminus_d}(u(x)-u(y))(v(x)-v(y))J(dx,dy)
    \end{multline*}
where $a^{(c)}$ is a strongly local symmetric form with domain
$D(a^{(c)})=D(a)\cap C_c(X)$, $J$ is a symmetric positive Radon
measure on $X\times X\setminus_d$ ($d=$diagonal), and $k$ is a
positive Radon measure on $X$. In addition such $a^{(c)},J$ and $k$
are uniquely determined by $a$.
\end{theorem}

The Beurling-Deny and Lejan formula had catched a lot of attention
last years and many generalization in other directions was explored.
For example in the case of semi-regular Dirichlet forms, or Regular
but non-symmetric Dirichlet forms. For more information we refer to
\cite{HM}, \cite{HMS1}, \cite{HMS2}  and \cite{MR} and references
therein. Here, we will use the above ``conventional'' Beurling-Deny
and Lejan formula that is the one for Regular Dirichlet forms, but
there is no reason that the same arguments do not work also for the
other cases.

\section{About an example in \cite{AW2}}

Let $\Omega\subset\RR^d$ be a bounded open set with Lipschitz
boundary $\Gamma:=\po$, $\sigma$ is the $(d-1)-$ dimensional
Hausdorff measure on $\Gamma$ and $B$ a bounded operator on
$L^2(\Gamma)$. Define the bilinear form $a_B$ with domain $\Hun$ on
$L^2(\Omega)$ by
\[
 a_B(u,v)=\int_{\Omega}\nabla u\nabla v dx+\int_{\Gamma}Bu_{|\Gamma}v_{|\Gamma}d\sigma
\]

The choice of the domain $D(a_B)=\Hun$ comes from the fact that $B$
is bounded in $L^2(\Gamma)$ and from the continuity of the embedding
$H^1(\Omega)\hookrightarrow L^2(\Gamma)$.

It is clear that the billinear form $(a_B,\Hun)$ is continuous and
elliptic. Let $-\Delta_B$ be the associated operator of the form
$a_B$, and we note $(e^{t\Delta_B})_{t\geq 0}$ the associated
semigroup. It is natural to expect that $-\Delta_B$ is a realization
of the Laplacian in $L^2(\Omega)$, which means that if we let
$u,f\in L^2(\Omega)$, then $u\in D(\Delta_B)$ and $-\Delta_B u=f$ if
and only if $u\in\Hun$, $-\Delta u=f$ and $\partial_{\nu}u+Bu=0$ on
$\Gamma$. The operator $-\Delta_B$ is then called the Laplacian with
nonlocal Robin boundary conditions. Finally, by \cite[C-II, Theorem
1.11. p:255]{Na} one can deduce easily that $e^{t\Delta_B}\geq 0$
for all $t\geq 0$ if and only if there exists a constant $c$ such
that $B-cI\leq 0$ on $L^2(\Gamma)$\cite{Ar}($I$ stands for the
identity operator on $L^2(\Gamma)$).

Now, we consider the case where $d=1$. Let $\Omega=(0,1)$ then
$\Gamma=\{0,1\}$ and $B=(b_{ij})_{i,j=1,2}\in\mathbb R^{2\times 2}$.
In this simple situation, one can see easily when the semigroup
$(e^{t\Delta_B})_{t\geq 0}$ is positive or not. In fact, using the
above characterization, one can conclude that $e^{t\Delta_B}\geq 0$
for all $t\geq 0$ if and only if $b_{ij} \leq 0$ for $i\neq j$, it
suffice to choose $c=b_{11}\vee b_{22}$(see also \cite[C-II, Example
1.13]{Na} for similar situations).

The Example 4.5 of \cite{AW2} corresponds to the choice
\[B=\begin{pmatrix} 1&1\\1&1\end{pmatrix}\] and gives,
\[
 a(u,v)=\int_0^1u'v'dx+u(0)v(0)+u(1)v(0)+u(0)v(1)+u(1)v(1)
\]fot all $u,v\in H^1(0,1)$. This example is used in \cite{AW1}
to show that the locality condition in Theorem 4.1 and Corollary 4.2
can not be omitted. The authors used Ouhabaz's domination criterion
to prove that the semigroup $e^{t\Delta_B}$ is sandwiched between
Dirichlet semigroup and Neumann semigroup, but $e^{t\Delta_B}$ is
obviously not positive, the situation where Ouhabaz's domination
criterion can not be used. In fact one can prove that any positive
semigroup dominated by Neumann semigroup is automatically local.

Consider then the situation of Theorem 4.1 in \cite{AW1}. Let $T$ be
a symmetric $C_0-$semigroup on $L^2(\Omega)$ associated with a
positive closed form $(a, D(a))$ such that
\[
 0\leq T(t)\leq e^{t\Delta^N},\quad (t\geq 0)
\]

Thus, we have, simultaneously, for $u\in D(a)$ that $a(u^+,u^-)\leq
0$ and $a(u^+,u^-)\geq \int_{\Omega}\nabla u^+\nabla u^- dx=0$,
which means that $a(u^+,u^-)=0$ for all $u\in D(a)$ and then that
$a$ is local\cite{Ar}. One can take also, instead of Neumann
semigroup, any symmetric $C_0-$ semigroup associated with a positive
closed local form $(b,D(b)$.

\begin{remark}
We established above that the semigroup in \cite[Example 4.5]{AW1}
can not be positive. This fact suggested to explore wether the
semigroup is or is not eventual positive. The answer is affirmative,
see \cite[Example 6.12]{DGK}.
\end{remark}

\section{The sandwiched property using bdl formula}

In this section we will prove that one can drop the assumption of
locality in results about sandwiched property. This can be made, for
example, by using the Beurling-Deny and Lejan decomposition formula
of regular Dirichlet forms. One can then obtain the following
result:

\begin{theorem}\label{thm:1}
 Let $\Omega$ be a bounded open set of  $\RR^d$ and $(T(t))_{t\geq 0}$ be a $C_0-$semigroup
 on $L^2(\Omega)$ associated with a regular Dirichlet form $(a,D(a))$. Then the following assertions are equivalent:
\begin{enumerate}

 \item  $a=a_{\mu}$ for a unique Radon measure $\mu$ on $\partial\Omega$ which charges no set of zero relative capacity.

\item  $e^{t\Delta_D}\leq T(t)\leq e^{t\Delta_N}$ for all $t\geq 0$.
\end{enumerate}
\end{theorem}
\begin{proof}

Let $(a,D(a)$ be the regular Dirichlet form associated with
$(T(t))_{t\geq 0}$. From the formula of Beurling-Deny and Lejan we
get for all $u,v\in D(a)\cap C_c(\overline{\Omega})$
\begin{multline*}
     a(u,v)=a^{(c)}(u,v)+\int_{\overline{\Omega}}u(x)v(x)k(dx)\\
     +\int_{\overline{\Omega}\times \overline{\Omega}\setminus_d}(u(x)-u(y))(v(x)-v(y))J(dx,dy)
    \end{multline*}
where $a^{(c)}$ is a strongly local form, $k$ a positive radon
measure on $\overline{\Omega}$ and $J$ a symmetric positive radon
measure on $\overline{\Omega}\times\overline{\Omega}\setminus_d$. To
simplify calculations, we note for all $u,v\in D(a)\cap
C_c(\overline{\Omega})$
\[
 a_k(u,v)=a^{(c)}(u,v)+\int_{\overline{\Omega}} u(x)v(x)k(dx),
\]and
\[
 b(u,v)=a(u,v)-a_k(u,v),
\]
We have then for all $u,v\in D(a)\cap C_c(\overline{\Omega}) $ such
that $\mathrm{supp}[u]\cap\mathrm{supp}[v]=\emptyset$
\begin{eqnarray}
 b(u,v) &=&  \int_{\overline{\Omega}\times \overline{\Omega}\setminus_d}(u(x)-u(y))(v(x)-v(y))J(dx,dy)     \nonumber \\[0,2cm]
   &=& 2\int_{\overline{\Omega}}u(x)v(x)J(dx,dy)-2\int_{\overline{\Omega}\times\overline{\Omega}\setminus_d}u(x)v(y)J(dx,dy)\\[0,2cm]
   &=&-2\int_{\overline{\Omega}\times\overline{\Omega}\setminus_d}u(x)v(y)J(dx,dy)
\end{eqnarray}
From the fact that $a_k$ is local, we get for all $u,v\in D(a)\cap
C_c(\overline{\Omega})$ such that
$\mathrm{supp}[u]\cap\mathrm{supp}[v]=\emptyset$
\[
 a(u,v)=-2\int_{\overline{\Omega}\times\overline{\Omega}\setminus_d}u(x)v(y)J(dx,dy)
\]
From \eqref{eq:san}, and the Ouhabaz's domination
criterion\cite[Th\'eor\`eme 3.1.7.]{Ou}, we have that $a_N(u,v)\leq
a(u,v)$ for all $u,v\in D(a)_+$, and then for all $u,v\in D(a)_+\cap
C_c(\overline{\Omega})$ such that
$\mathrm{supp}[u]\cap\mathrm{supp}[v]=\emptyset$ we have
\[
 0\leq -2\int_{\overline{\Omega}\times\overline{\Omega}\setminus_d}u(x)v(y)J(dx,dy),
\]
which means that
$-2\int_{\overline{\Omega}\times\overline{\Omega}\setminus_d}u(x)v(y)J(dx,dy)=0$.
Thus $\mathrm{supp}[J]\subset d$ and then for all $u,v\in D(a)\cap
C_c(\overline{\Omega})$
\[
\int_{\overline{\Omega}\times\overline{\Omega}\setminus_d}u(x)v(y)J(dx,dy)=\int_{\overline{\Omega}}u(x)v(x)J(dx,dy)
\]
Then, we deduce that $b(u,v)=0$ for all $u,v\in D(a)\cap
C_c(\overline{\Omega})$ and thus the form $a$ is immediately local
and is reduced to the following
\[
 a(u,v)=a^{(c)}(u,v)+\int_{\overline{\Omega}} u(x)v(x)k(dx),\text{ for all }u,v\in D(a)\cap C_c(\overline{\Omega})
\]
We have $C_c^{\infty}(\Omega)\subset D(a)$, and then for all $u,v\in
C_c^{\infty}(\Omega)$ we have from \eqref{eq:san}
\[
 \int_{\Omega}\nabla u\nabla v dx\leq a^{(c)}(u,v)+\int_{\Omega} u(x)v(x)k(dx)\leq  \int_{\Omega}\nabla u\nabla v dx
\]
It follows that for all $u,v\in C_c^{\infty}(\Omega)$, we get
\begin{eqnarray}
 a(u,v) &=& a^{(c)}(u,v)+\int_{\Omega} u(x) v(x) k(dx) \\[0,2cm]
   &=& \int_{\Omega}\nabla u\nabla v dx
\end{eqnarray}
Thus $\mathrm{supp}[k]\subset\partial\Omega$ and then by putting
$\mu=k_{|\partial\Omega}$, we have for all  $u,v\in D(a)\cap
C_c(\overline{\Omega})$
\[
 a(u,v)=\int_{\Omega}\nabla u\nabla v dx+\int_{\partial\Omega}u(x)v(x)\mu(dx)
\]
For the rest of the proof, one can follow exactly the end of the
proof of \cite[Corollary 3.4.23]{Wa}
\end{proof}
One can see easily that the locality property was automatic in the
above proof and that just the right hand side domination property
was needed for it. We give then the following Corollary.
\begin{corollary}\label{cor:1}
  Let $\Omega$ be a bounded open set of  $\RR^d$ and $(T(t))_{t\geq 0}$ be a $C_0-$semigroup
  on $L^2(\Omega)$ associated with a regular Dirichlet form $(a,D(a))$ such that
  $0\leq T(t)\leq e^{t\Delta_N}$ for all $t\geq 0$. Then the regular Dirichlet form $(a,D(a))$ is local.
\end{corollary}
\begin{proof}
 The corollary come directly from the proof of Theorem \ref{thm:1}.
\end{proof}
We deduce from the above discussion the following general theorem
\begin{theorem}
 Let $X$ be locally compact separable metric space and $m$ a positive Radon measure on $X$ such
 that $\mathrm{supp}[m]=X$. Let $(T(t))_{t\geq 0}$ (resp. $(S(t))_{t\geq 0}$) be a $C_0-$semigroup
 on $L^2(X;m)$ associated with a regular Dirichlet form $(a,D(a))$ (resp. with a Dirichlet form $(b,D(b))$). Assume
 in addition that $T(t)$ is subordinated by $S(t)$ that is $T(t)\leq S(t)$ for all $t\geq 0$ in the positive
 operators sense. Then $b$ is local implies $a$ is local.
\end{theorem}
\begin{proof}
 The proof is based on the same argumets as for Corollary \ref{cor:1}. In fact, we have that $b\leq a$ and $b$ is
 local then for all $u,v\in D(a)_+\cap C_c(X)$ with disjoint support we have $a(u,v)\geq 0$. Thus, by using
 Beurling-Deny and Lejan formula and the positivity of $u$ and $v$ one obtain
\[
 \int_{X\times X\setminus_d}u(x)v(y)J(dx,dy)=0
\]for all $u,v\in D(a)\cap C_c(X)$ with disjoint support. Thus $\mathrm{supp}[J]\subset d$, and then the jump integral
in the Beurling-Deny and Lejan decomposition of the form $(a,D(a))$ is null, which achieve the proof.

\end{proof}
\begin{corollary}
  Let $\Omega$ be a bounded open set of  $\RR^d$ and $(T(t))_{t\geq 0}$ be a $C_0-$semigroup on $L^2(\Omega)$ associated
  with a Dirichlet form $(a,D(a))$. Then the following assertions are equivalent:
\begin{enumerate}
 \item[(i)] $a=a_{\mu}$ for some positive measure $\mu$ on $\partial\Omega$ which charges no set of zero relative capacity
 on which it is locally finite.
 \item[(ii)] (a) $e^{t\Delta_D}\leq T(t)\leq e^{t\Delta_N}$ for all $t\geq 0$

       (b) $D(a)\cap C_c(\overline{\Omega})$ is dense in $D(a)$.
\end{enumerate}
\end{corollary}
\begin{proof}
Let
\[
 \Gamma_0=\{z\in\Gamma:\exists u\in D(a)\cap C_c(\overline{\Omega}), u(z)\neq 0\}
\]
From the Stone-Weierstrass theorem, one can see that the form
$(a,D(a))$ is regular on $L^2(X)$ where $X=\Omega\cup\Gamma_0$, one
can see it exactly from the proof in \cite[Theorem 4.1]{AW1} or the
one in \cite[Theorem 3.4.21]{Wa}. Following the same procedure as in
Theorem \ref{thm:1}, we obtain that there exist a unique positive
Radon measure $k$ on $X$ such that
\[
 a(u,v)=a^{(c)}(u,v)+\int_Xu(x)v(x)k(x)
\]
for all $u,v\in D(a)\cap C_c(X)$.

Now one can just again follow the same steps as in the proof of
\cite[Theorem 4.1]{AW1} or the one in \cite[Theorem 3.4.21]{Wa}
\end{proof}

\begin{proposition}
 Let $\Omega\subset\RR^d$ be a bounded open set of class $C^1$ with boundary $\Gamma$. Let $T$ be a symmetric $C_0-$semigroup
 associated with a regular Dirichlet form $(a,D(a))$. Denote by $A$ the generator of $T$. Assume that

a) $e^{t\Delta_D}\leq T(t)\leq e^{t\Delta_N}$ for all $t\geq 0$, and
that

b) there exists $u\in D(A)\cap C^2(\overline{\Omega})$ such that
$u(z)>0$ for all $z\in\Gamma$. Then there exists a function
$\beta\in C(\Gamma)_+$ such that $A=-\Delta_{\beta}$.
\end{proposition}

\begin{proof}
 The proof is exactly the same as \cite[Proposition 5.2]{AW1}, we have just omitted the locality hypothesis in the proposition.
\end{proof}

In what follow we give some complements and remarks about our
approach and the one of W. Arendt and M. Warma.
\begin{itemize}
 \item Apparently the both methods, ours and the one in \cite{AW1} seem to be different, but
 in fact there is some connection between them. Seeing things more in details in the proofs
 of \cite[Theorem 4.1]{AW1}, and the proof of the formula of Beurling-Deny and Lejan, one
 can see that they make use of \cite[Lemma 1.4.1]{FOT}(or more exactly its proof) as a central tool to find a measure.

  \item The ideal property was not used in the sufficient implication of Theorem \ref{thm:1}.

  \item One can expect to generalize our results without difficulty to the semi-regular Dirichlet
  forms and to regular but non-symmetric regular Dirichlet forms.

  \item In the context of $p-$Laplacian operator R. Chill and M. Warma proved a version of the
   \cite[Theorem 4.1. and Theorem 4.2]{AW1} see \cite{CW}. Unfortunately, there is no version
   of the formula of Beurling-Deny and Lejan in $L^p(X,m)$, see \cite{Bir} as a beginning of interest.

  \item It is proved in an unpublished note of W. Arendt and M. Warma that a locality of
  the Dirichlet form implies the locality of the associated operator, but the converse is not
  true. We can see easily from our results that the operator is local if and only
  if $J_{|_{\Omega\times\Omega\setminus_d}}=0$, which means that there is an equivalence between
  the locality of the operator and the one of the form when there is no jump inside the domain.
\end{itemize}

\par\bigskip

\section*{Acknowledgments} The author would like to thank Wolfgang Arendt Omar El-Mennaoui
and Jochen Gl\"uck for many stimulating and helpful discussions.
This work was achieved during a research stay in Ulm, Germany.

\end{document}